\documentclass{amsart}
\usepackage{amssymb, amsmath, latexsym, mathabx, dsfont, enumerate, multirow}

\renewcommand{\baselinestretch}{\baselinestretch}
\renewcommand{\baselinestretch}{1.1}
\numberwithin{equation}{section}

\newtheorem{thm}{Theorem}[section]

\newcommand{\Mod}[1]{\ (\mathrm{mod}\ #1)}

\begin{document}

\title{Even universal sums of triangular numbers}
\author{Jangwon Ju}

\address{Department of Mathematics Education, Korea National University of Education, Cheongju-si, Chungbuk 28173, Republic of Korea}
\email{jangwonju@knue.ac.kr}
\thanks{}



\begin{abstract}
For an arbitrary integer $x$, an integer of the form $T(x)\!=\!\frac{x^2+x}{2}$ is called a triangular number.
Let $\alpha_1,\dots,\alpha_k$ be positive integers. A sum $\Delta_{\alpha_1,\dots,\alpha_k}(x_1,\dots,x_k)=\alpha_1 T(x_1)+\cdots+\alpha_k T(x_k)$ of triangular numbers is said to be {\it even universal} if the Diophantine equation $\Delta_{\alpha_1,\dots,\alpha_k}(x_1,\dots,x_k)=2n$ has an integer solution $(x_1,\dots,x_k)\in\mathbb{Z}^k$ for any nonnegative integer $n$.
In this article, we classify all even universal sums of triangular numbers. 
Furthermore, we provide an effective criterion on even universality of an arbitrary sum of triangular numbers, which is a generalization of the triangular theorem of eight of Bosma and Kane.
\end{abstract}

\maketitle

\section{Introduction}
For an arbitrary integer $x$, an integer of the form $T(x)\!=\!\frac{x^2+x}{2}$ is called a triangular number.
For positive integers $\alpha_1,\dots,\alpha_k$, we say a sum 
$$
\Delta_{\alpha_1,\dots,\alpha_k}(x_1,\dots,x_k):=\alpha_1 P_3(x_1)+\cdots+\alpha_k P_3(x_k)
$$ 
of triangular numbers {\it represents} a nonnegative integer $n$ if the diophantine equation 
$$
\Delta_{\alpha_1,\dots,\alpha_k}(x_1,\dots,x_k)=n
$$
has an integer solution $(x_1,\dots,x_k)\in\mathbb{Z}^k$. 
Furthermore, a sum
$$
\Delta_{\alpha_1,\dots,\alpha_k}(x_1,\dots,x_k) \ (\text{simply}, \ \Delta_{\alpha_1,\dots,\alpha_k}) 
$$
of triangular numbers is called {\it universal} if it represents all nonnegative integers.

 In 1796, Gauss proved that every positive integer can be expressed as a sum of three triangular numbers which was first asserted by Fermat in 1638. 
In 1862, Liouville proved that for positive integers $a,b$, and $c~ (a\leq b\leq c)$, a ternary sum  $\Delta_{a,b,c}$ of triangular numbers is universal if and only if $(a,b,c)$ is one of the following triples:
\begin{equation}\label{Liouville}
(1,1,1),~ (1,1,2),~ (1,1,4),~ (1,1,5),~ (1,2,2),~ (1,2,3),~\text{and}~ (1,2,4).
\end{equation}

In 2013, Bosma and Kane proved the triangular theorem of eight which states that for positive integers $\alpha_1,\dots,\alpha_k$, an arbitrary sum $\Delta_{\alpha_1,\dots,\alpha_k}$ of triangular numbers is universal if and only if it represents $1,2,4,5$, and $8$ (for details, see \cite{BK}). 
This might be considered as a natural generalization of the ``15-theorem'' of Conway, Miller, and Schneeberger (for details see \cite{B}).

 If a sum $\Delta_{\alpha_1,\dots,\alpha_k}$ of triangular numbers represents all nonnegative integers except a single one, then it is said to be {\it almost universal with one exception}.
By the triangular theorem of eight of Bosma and Kane, if a sum $\Delta_{\alpha_1,\dots,\alpha_k}$ of triangular numbers is almost universal with one exception $m$, then $m$ is 1,2,4,5, or 8. 
In \cite{almost}, the author classified all almost universal sums of triangular numbers with one exception $1,2,4,5$, and $8$, respectively.

In 2009, Kane conjectured that  for positive integers $\alpha_1,\dots,\alpha_k$, an arbitrary sum $\Delta_{\alpha_1,\dots,\alpha_k}$ of triangular numbers  represents all positive odd integers if and only if it represents the integers
$$
1,\ 5,\ 7,\ 9,\ 11,\ 13,\ 17,\ 19,\ 25,\ 29,\ 35,\ 49,\ \text{and} \ 89.
$$
Actually, he proved the above theorem under the assuming GRH for $L$-functions of weight $2$ newforms (for details, see \cite{odd universal}). 

For positive integers $\alpha_1,\dots,\alpha_k$, a sum $\Delta_{\alpha_1,\dots,\alpha_k}$ of triangular numbers is called {\it even universal} if it represents all nonnegative even integers.
Furthermore, it is called {\it proper} if any proper partial sum of it doesn’t represent at least one nonnegative even integers.
By using an escalation method, we give a complete list of candidates of proper even universal sums of triangular numbers. 
Furthermore, we classify all even universal sums of triangular numbers, actually there are 15 ternary, 37 quaternary and 23 quinary sums of proper even universal sums of triangular numbers (see Table 1). 
Moreover, we provide an effective criterion on even universality of an arbitrary sum $\Delta_{\alpha_1,\dots,\alpha_k}$ of triangular numbers.
This might be considered as a natural generalization of the triangular theorem of eight of Bosma and Kane.
\begin{thm}\label{1}
A sums of triangular numbers is even universal if and only if it represents the integers
$$
2,\ 4,\ 8,\ 10,\ 14,\ 16,\ 26,\ \text{and}\ 40.
$$
\end{thm}

Let
$\displaystyle
f(x_1,x_2,\dots,x_k)=\sum_{1 \le i, j\le k} a_{ij} x_ix_j \ (a_{ij}=a_{ji} \in \mathbb Z)
$
 be a positive definite integral quadratic form. 
For an integer $n$, we say $n$ is {\it represented by $f$} if the equation $f(x_1,x_2,\dots,x_k)=n$ has an integer solution $(x_1,x_2,\dots,x_k)\in\mathbb{Z}^k$, which is denoted by $n\longrightarrow f$. 
The genus of $f$, denoted by $\text{gen}(f)$, is the set of all quadratic forms that are isometric to $f$ over the $p$-adic integer ring $\mathbb{Z}_p$ for any prime $p$. 
The number of isometry classes in $\text{gen}(f)$ is called the class number of $f$.

A good introduction to the theory of quadratic forms may be found in \cite{om}, and we adopt the notations and terminologies from this book.


\section{General tools}
For positive integers $\alpha_1,\dots,\alpha_k~(k\geq1)$, we define
$$
\Delta_{\alpha_1,\dots,\alpha_k}(x_1,\dots,x_k)=\alpha_1T(x_1)+\cdots+\alpha_kT(x_k).
$$
Recall that a sum 
$$
\Delta_{\alpha_1,\dots,\alpha_k}(x_1,\dots,x_k) \ (\text{simply}, \ \Delta_{\alpha_1,\dots,\alpha_k})
$$ 
of triangular numbers is called  even universal if it represents all nonnegative even integers,
which is equivalent to the existence of an integer solution $(x_1,\dots,x_k)\in\mathbb{Z}^k$ of 
$$
\alpha_1(2x_1+1)^2+\cdots+\alpha_k(2x_k+1)^2=16n+\alpha_1+\cdots+\alpha_k 
$$
for any nonnegative integer $n$.
Furthermore, this is equivalent to the existence of an integer solution $(x_1,\dots,x_k)\in\mathbb{Z}^k$ of  
\begin{equation}\label{congruence condition}
\alpha_1x_1^2+\cdots+\alpha_kx_k^2=16n+\alpha_1+\cdots+\alpha_k~\text{with}~x_1\cdots x_k\equiv1\Mod2
\end{equation}
for any nonnegative integer $n$.

 In Section 3, we show that there are no unary or binary even universal sums of triangular numbers. 
 Furthermore, we prove that there are exactly 15 ternary even universal sums of triangular numbers. 
 In most cases,  their even universalities was proved by the result of Liouville.

To prove the even universality of a sum of triangular numbers $\Delta_{\alpha_1,\dots,\alpha_k}$ $(k\geq4)$, we use the similar strategy used to prove the almost universality of a sum of triangular numbers in \cite{almost}.
We briefly introduce this strategy for those who are unfamiliar with it.


Without loss of generality, we may assume that $\alpha_1\leq\cdots\leq\alpha_k$.
At first, take the ternary section $\Delta_{\alpha_1,\alpha_2,\alpha_3}$ of $\Delta_{\alpha_1,\dots,\alpha_k}$.
We consider the equation
\begin{equation}\label{ternary form}
\alpha_1 x_1^2+\alpha_2 x_2^2+\alpha_3 x_3^2=16n+\alpha_1+\alpha_2+\alpha_3 ~\text{with}~ x_1x_2x_3\equiv1\Mod2.
\end{equation}
Note that Equation \eqref{ternary form} corresponds to the representations by a ternary quadratic form with congruence conditions. 
Work of Oh \cite{regular, pentagonal} and work of Oh and the author \cite{4-8} led to the development of a method that determines whether or not integers in an arithmetic progression are represented by some particular ternary quadratic form.
Therefore, we try to find a suitable method on reducing Equation \eqref{ternary form} to the representations of a ternary quadratic form, denoted by $f(x_1,x_2,x_3)$, without congruence conditions.
To explain our method, for example, assume that $\alpha_1\equiv\alpha_2\equiv0\Mod2$ and $\alpha_3 \equiv1\Mod2$. 
Then Equation \eqref{ternary form} has an integer solution if 
\begin{equation}\label{without congruence condition}
f(x_1,x_2,x_3)=\alpha_1(x_3-2x_1)^2+\alpha_{2}(x_3-2x_{2})^2+\alpha_3x_3^2=16n+\alpha_1+\alpha_2+\alpha_3
\end{equation}
has an integer solution. 
Hence, in this case, the problem can be reduced to the representations of a ternary quadratic form without congruence conditions.

After that for sufficiently large $n$, we find suitable $a_4,\dots,a_k\in\mathbb{Z}$ such that
\begin{equation}\label{conditions}
\begin{array}{rl}
\rm{(i)}   &  a_4\cdots a_k\equiv1\Mod2; \\
\rm{(ii)}  &  16n+\alpha_1+\cdots+\alpha_k-(\alpha_4 a_4^2+\cdots+\alpha_k a_k^2)\geq0; \\
\rm{(iii)} &  16n+\alpha_1+\cdots+\alpha_k-(\alpha_4 a_4^2+\cdots+\alpha_k a_k^2) \longrightarrow f(x_1,x_2,x_3). \\
\end{array}
\end{equation}
Then we know that Equation \eqref{congruence condition} has an integer solution.
Finally, we directly check that the sum $\Delta_{\alpha_1,\alpha_2,\dots,\alpha_k}$ of triangular numbers represents all remaining small even integers.

\section{Sums of triangular numbers representing all even integers}

Let $k\geq 1$. For positive integers $\alpha_1,\dots,\alpha_k$, we say an even universal sum $\Delta_{\alpha_1,\dots,\alpha_k}$ of triangular numbers is {\it proper} if for any proper subset $\{\alpha_{i_1},\dots,\alpha_{i_u}\}$ of $\{\alpha_1,\dots,\alpha_k\}$, the partial sum $\Delta_{\alpha_{i_l},\dots,\alpha_{i_u}}$ doesn't represent at least one nonnegative even integers. 

 The first nonnegative (nonnegative even) integer that is not represented by $\Delta_{\alpha_1,\dots,\alpha_k}$ is called the {\it  truant} ({\it  even truant}) of $\Delta_{\alpha_1,\dots,\alpha_k}$ and denoted by 
 $$
 \mathfrak{T}(\Delta_{\alpha_1,\dots,\alpha_k}) \quad (\mathfrak{T}_e(\Delta_{\alpha_1,\dots,\alpha_k}), \text{ respectively}) 
 $$
 if it exists.
If a sum $\Delta_{\alpha_1,\dots,\alpha_k}$ of triangular numbers is even universal, then we define $\mathfrak{T}_e(\Delta_{\alpha_1,\dots,\alpha_k})=\infty$.
We say every positive integer is less than $\infty$ for convenience.

For a sum $\Delta_{\alpha_1,\dots,\alpha_k}$ of triangular numbers, without loss of generality, we may assume that 
$\alpha_1\leq\cdots\leq\alpha_k$.\
We say $\Delta_{\alpha_1,\dots,\alpha_k}$ is a {\it candidate} of even universal sums of triangular numbers if it satisfies the following conditions:
\begin{enumerate}
\item[(i)]  $\alpha_1 \leq 2$;
\item[(ii)] if $k\geq2$, then $\alpha_i\leq \mathfrak{T}_e(\Delta_{\alpha_1,\dots,\alpha_{i-1}})  \text{~for~} 2\leq i \leq k$.
\end{enumerate}
Assume that $\Delta_{\alpha_1,\dots,\alpha_k}(\alpha_1\leq\dots\leq\alpha_k)$ is even universal.
Since it represents $2$, we know that $\alpha_1 \leq 2$.
When $k\geq2$, for $2\leq i\leq k$, if $\Delta_{\alpha_1,\dots,\alpha_{i-1}}$ is not even universal, then  $\alpha_i\leq \mathfrak{T}_e(\Delta_{\alpha_1,\dots,\alpha_{i-1}})$ since it represents $\mathfrak{T}_e(\Delta_{\alpha_1,\dots,\alpha_{i-1}})$.
If $\Delta_{\alpha_1,\dots,\alpha_{i-1}}$ is even universal, then $\alpha_i\leq \mathfrak{T}_e(\Delta_{\alpha_1,\dots,\alpha_{i-1}})$ since $\mathfrak{T}_e(\Delta_{\alpha_1,\dots,\alpha_{i-1}})=\infty$.
Therefore, every even universal sum of triangular numbers satisfies the above conditions.

Let $\alpha_1,\dots,\alpha_k$ be positive integers with $\alpha_1\leq\cdots\leq\alpha_k$. Assume that a sum $\Delta_{\alpha_1,\dots,\alpha_k}$ of triangular numbers is a candidate of even universal sums of triangular numbers.
When $k=1$, we know that $1\leq\alpha_1\leq2$  from the definition of the candidate of even universal sums of triangular numbers.
One may easily check that
\begin{equation}\label{unary ET}
\mathfrak{T}_e(\Delta_1)=2 \quad\text{and}\quad \mathfrak{T}_e(\Delta_2)=4. 
\end{equation}
Therefore, there are no unary even universal sums of triangular numbers.

Let $k=2$. 
Form \eqref{unary ET}, we have five candidates 
$$
\Delta_{1,1}, \ \Delta_{1,2}, \ \Delta_{2,2}, \ \Delta_{2,3}, \text{ and } \Delta_{2,4}
$$
of binary even universal sums of triangular numbers.
One may easily check that
\begin{equation}\label{binary ET}
\mathfrak{T}_e(\Delta_{\alpha_1,\alpha_2})=
\begin{cases}
{\setlength\arraycolsep{1pt}
\begin{array}{ll}
8~ &\text{if}~ (\alpha_1,\alpha_2)=(1,1),\\
4~ &\text{if}~ (\alpha_1,\alpha_2)=(1,2),\\
10~ &\text{if}~ (\alpha_1,\alpha_2)=(2,2),\\
4~ &\text{if}~ (\alpha_1,\alpha_2)=(2,3),\\
8~ &\text{if}~ (\alpha_1,\alpha_2)=(2,4).\\
\end{array}}
\end{cases}
\end{equation}
Therefore, there are no binary even universal sums of triangular numbers.

Let $k=3$.
From, \eqref{binary ET}, we know that there are 27 candidate $\Delta_{\alpha_1,\alpha_2,\alpha_3}$ of even universal sums of triangular numbers since $\alpha_3\leq \mathfrak{T}_e(\Delta_{\alpha_1,\alpha_2})$ in each possible case.
One may easily check that 12 sums of them don't represent at least one even integer.
Actually, one may easily check that
\begin{equation}\label{ternary ET}
\mathfrak{T}_e(\Delta_{\alpha_1,\alpha_2,\alpha_3})=
\begin{cases}
{\setlength\arraycolsep{1pt}
\begin{array}{llllll}
&\mathfrak{T}_e(\Delta_{\alpha_1,\alpha_2}) ~&\text{if}&~(\alpha_1,\alpha_2,\alpha_3)&=&(1,1,3),(2,2,3),(2,2,5),(2,2,7)\\
&~&~&~&~&(2,2,9),(2,3,3),(2,4,5),(2,4,7)\\ 
&14 ~ &\text{if}&~(\alpha_1,\alpha_2,\alpha_3)&=&(1,1,6),\\
&26 ~ &\text{if}&~(\alpha_1,\alpha_2,\alpha_3)&=&(1,1,7),\\ 
&16 ~ &\text{if}&~(\alpha_1,\alpha_2,\alpha_3)&=&(2,2,6),\\
&8 ~ &\text{if}&~(\alpha_1,\alpha_2,\alpha_3)&=&(2,3,4).\\  
\end{array}}
\end{cases}
\end{equation}

Now, we prove that remaining 15 ternary candidates represent all even integers.
If $(\alpha_1,\alpha_2,\alpha_3)$ is contained in the seven triples of \eqref{Liouville}, then clearly $\Delta_{\alpha_1,\alpha_2,\alpha_3}$ and  $\Delta_{2\alpha_1,2\alpha_2,2\alpha_3}$ are even universal. 
Let $(\alpha_1,\alpha_2,\alpha_3)=(1,1,8)$.
By Equation \eqref{congruence condition}, it suffices to show that the equation
\begin{equation}\label{118}
x^2+y^2+8z^2=16n+10
\end{equation}
has an integer solution $(x,y,z)\in\mathbb{Z}^3$ such that $xyz\equiv1\Mod2$ for any nonnegative integer $n$.
Let $f(x,y,t)=x^2+(2y+z)^2+8z^2$.
For any nonnegative integer $n$, $16n+10$ is represented by $f$ since it is not a sinor exceptional integer of $\text{gen}(f)$
(for details, see Section 5 of \cite{almost}).
Therefore, Equation \eqref{118} has an integer solution $(x,y,z)\in\mathbb{Z}^3$ such that $xyz\equiv1\Mod2$ for any nonnegative integer $n$.

Now, assume that $k\geq4$. 
 Note that the ternary section $\Delta_{\alpha_1,\alpha_2,\alpha_3}$ of $\Delta_{\alpha_1,\dots,\alpha_k}$  is one of the above 27 ternary candidates of even universal sums of triangular numbers.
 If $\mathfrak{T}_e(\Delta_{\alpha_1,\alpha_2,\alpha_3})=\infty$, then it implies that $\alpha_4,\dots,\alpha_k$ can be any integers with $\alpha_3\leq \alpha_4\leq\cdots\leq\alpha_k$. 
 In this cases, $\Delta_{\alpha_1,\alpha_2,\alpha_3\alpha_4,\dots\alpha_k}$ is even universal but not proper.
 Assume that $\mathfrak{T}_e(\Delta_{\alpha_1,\alpha_2,\alpha_3})\neq\infty$. 
In this cases, we classify all proper even universal sums $\Delta_{\alpha_1,\alpha_2,\alpha_3\alpha_4,\dots\alpha_k}$ of triangular numbers according to its ternary section $\Delta_{\alpha_1,\alpha_2,\alpha_3}$.

\vskip1pc
\noindent\textbf{(i)} Let $(\alpha_1,\alpha_2,\alpha_3)=(1,1,3)$. For $k=4$, by the definition of the candidate of even universal sums of triangular numbers, we know that $3\leq \alpha_4 \leq 8$.
If $\alpha_4=4,5,8$, then $\Delta_{1,1,3,\alpha_4}$ is even universal but not proer.
If $\alpha_4=3,6,7$, then one may easily show that $\Delta_{1,1,3,\alpha_4}$ is universal by the triangular theorem of eight of Bosma and Kane. 

For $k\geq5$, all candidate $\Delta_{1,1,3,\alpha_4,\dots,\alpha_k}$ of even universal sums of triangular numbers are even universal but not proper since $\mathfrak{T}_e( \Delta_{1,1,3,\alpha_4})=\infty$.  
Therefore, if $k\geq 5$, then there are no $k$-ary proper even universal sums of triangular numbers.

 \vskip1pc
\noindent\textbf{(ii)}  Let $(\alpha_1,\alpha_2,\alpha_3)=(1,1,6)$.  
For $k=4$, by  the definition of the candidate of even universal sums of triangular numbers, we know that $6\leq \alpha_4 \leq 14$.
If $\alpha_4=8$, then $\Delta_{1,1,6,8}$ is even universal but not proer.
For $\alpha_4=7,10,11,12,13,14$,  in \cite{almost} it was proved that each $\Delta_{1,1,6,\alpha_4}$ represents all nonnegative integers except $5$. 

Assume that $\alpha_4=6$.
We show that $\Delta_{1,1,6,6}$ is even universal.
By Equation \eqref{congruence condition}, it suffices to show that the equation
\begin{equation}\label{1166}
x^2+y^2+6z^2+6t^2 = 16n+14
\end{equation}
has an integer solution $(x,y,z,t)\in\mathbb{Z}^4$ such that  $xyzt\equiv1 \Mod2$ for any nonnegative integer $n$.
The class number of $f(x,y,z)=x^2+y^2+6z^2$ is one.
From Table 11 of  \cite{almost}, for a nonnegative integer $m$, if $m\equiv0\Mod{8}$ and $m\neq3^{2u+1}(3v+1)$ for any nonnegative integers $u$ and $v$, then $m$ is represented by $f$.
Let $16n+14=3^{2\ell}(16k+14)$ for some nonnegative integers $\ell$ and $k$ such that $16k+14\not\equiv0\Mod{9}$.
If $0\leq k\leq2$,  then the equation
$$
x^2+y^2+6z^2+6t^2=16k+14
$$
has an integer solution $(x,y,z,t)\in\mathbb{Z}^4$ such that $xyzt\equiv1\Mod2$.
Assume that $k\geq3$.
One may easily check that $16k+14-6d^2$ is represented by $f$ over $\mathbb{Z}_p$ for any $p$, where
$$
d=\begin{cases}
3\quad\text{if}~16k+14\equiv6\Mod9,\\
1\quad\text{otherwise},
\end{cases}
$$
in particular, it is primitively represented by $f$ over $\mathbb{Z}_2$. 
Furthermore, since we are assuming $k\geq3$, $16k+14-6d^2$ is positive.
By 102:5 of \cite{om}, the equation 
$$
x^2+y^2+6z^2=16k+14-6d^2
$$ 
has an integer solution $(x,y,z)\in\mathbb{Z}^3$ such that $xyz\equiv1\Mod2$. 
Therefore, $\Delta_{1,1,6,6}$ is even universal.

Assume that $\alpha_4=9$.
Since $\mathfrak{T}_e(\Delta_{1,1,6,9})=14$, $\Delta_{1,1,6,9}$ is not an even universal sum of triangular numbers.

For $k\geq5$, if $\Delta_{1,1,6,\alpha_4,\dots,\alpha_k}$ is a candidate of even universal sums of triangular numbers, then one of the following holds:
\begin{enumerate}
\item $\mathfrak{T}_e(\Delta_{1,1,6,\alpha_4})=\infty$;
\item $\alpha_4=\cdots=\alpha_{k-1}=9$ and $10\leq\alpha_k\leq14$;
\item $\alpha_4=\cdots=\alpha_k=9$.
\end{enumerate}
In the first case, clearly $\Delta_{1,1,6,\alpha_4,\dots,\alpha_k}$ is even universal but not proper.
In the second case, since $\mathfrak{T}_e(\Delta_{1,1,6,\alpha_k})=\infty$, $\Delta_{1,1,6,\alpha_4,\dots,\alpha_k}$ is even universal but not proper .
In the third case, $\mathfrak{T}_e(\Delta_{1,1,6,\alpha_4\dots,\alpha_k})=14$ since  $\mathfrak{T}(\Delta_{1,1,6,\alpha_4\dots,\alpha_{k-1}})=5$ and $\mathfrak{T}_e(\Delta_{1,1,6,\alpha_4\dots,\alpha_{k-1}})=14$. 
Therefore, if $k\geq 5$, then there are no $k$-ary proper even universal sums of triangular numbers.
\vskip1pc
\noindent\textbf{(iii)}  Let $(\alpha_1,\alpha_2,\alpha_3)=(1,1,7)$.  
For $k=4$, by the definition of the candidate of even universal sums of triangular numbers, we know that $7\leq \alpha_4 \leq 26$.
If $\alpha_4=8$, then $\Delta_{1,1,7,8}$ is even universal but not proer.
Assume $7\leq \alpha_4 \leq 26$ and $\alpha_4\neq7,8,14,21$.  
In \cite{almost}, it was proved that each $\Delta_{1,1,7,\alpha_4}$ represents all nonnegative integers except $5$.

Assume that $\alpha_4=7$.
We show that $\Delta_{1,1,7,7}$ is even universal.
By Equation \eqref{congruence condition}, it suffices to show that the equation
\begin{equation}\label{1166}
x^2+y^2+7z^2+7t^2 = 16n+16
\end{equation}
has an integer solution $(x,y,z,t)\in\mathbb{Z}^4$ such that  $xyzt\equiv1 \Mod2$ for any nonnegative integer $n$. 
Let $f(x,y,z)=(2x+y)^2+y^2+7z^2$.
From Table 12 of \cite{almost}, for a nonnegative integer $m$, if $m>1$, $m\equiv1\Mod8$, $m\not\equiv0\Mod{49}$ and $m\neq7^{2u+1}(7v+r)$ for $r\in\{3,5,6\}$ and any nonnegative integers $u$ and $v$, then $m$ is represented by $f$.
Let $16n+16=7^{2\ell}(16k)$ for some nonnegative integers $\ell$ and $k$ such that $16k\not\equiv0\Mod{49}$.
If $1\leq k\leq21$,  then the equation
$$
x^2+y^2+7z^2+7t^2=16k
$$
has an integer solution $(x,y,z,t)\in\mathbb{Z}^4$ such that $xyzt\equiv1\Mod2$.
Assume that $k\geq22$.
One may easily check that  $16k-7d^2$ is positive and represented by $f$, where 
$$
d=
\begin{cases}
7 \quad\text{if}~ 16k\equiv7,14,28 \Mod{49},\\
3 \quad\text{if}~ 16k\equiv42 \Mod{49},\\
1 \quad\text{otherwise}.
\end{cases}
$$ 
Therefore,  the equation
$$
x^2+y^2+7z^2+7t^2=16k
$$
has an integer solution $(x,y,z,t)\in\mathbb{Z}^4$ such that $xyzt\equiv1\Mod2$.

Assume that $\alpha_4=14,21$.
Since $\mathfrak{T}_e(\Delta_{1,1,7,14})=40$ and $\mathfrak{T}_e(\Delta_{1,1,7,21})=26$ , $\Delta_{1,1,7,14}$ and $\Delta_{1,1,7,21}$ are not even universal.

For $k=5$, if $\Delta_{1,1,7,\alpha_4,\alpha_5}$ is a candidate of even universal sums of triangular numbers, then one of the following holds:
\begin{enumerate}
\item $\mathfrak{T}_e(\Delta_{1,1,7,\alpha_4})=\infty$;
\item $\alpha_4=14$ and $14\leq\alpha_5\leq40$;
\item $\alpha_4=21$ and $21\leq\alpha_5\leq26$.
\end{enumerate}
 In the first case, clearly $\Delta_{1,1,7,\alpha_4}$ is even universal but not proper. 
 In the second case, if $15\leq\alpha_5\leq20$ or $22\leq\alpha_5\leq26$, then $\Delta_{1,1,7,14,\alpha_5}$ is even universal but not proper since $\mathfrak{T}_e(\Delta_{1,1,7,\alpha_5})=\infty$. 
In \cite{almost}, we proved that if $\alpha_5=14,21$ or $27\leq\alpha_5\leq40$ and $\alpha_5\neq35$, then  $\Delta_{1,1,7,14,\alpha_5}$ represents all nonnegative integers except $5$. 
If $\alpha_5=35$, then one may easily check that $\mathfrak{T}_e(\Delta_{1,1,7,14,35})=40$.
In the third case, if $22\leq\alpha_5\leq26$, then $\Delta_{1,1,7,21,\alpha_5}$ is even universal but not proper.
If $\alpha_5=21$, then one may easily check that $\mathfrak{T}_e(\Delta_{1,1,7,21,21})=26$.

For $k\geq6$, if $\Delta_{1,1,7,\alpha_4,\alpha_5,\dots,\alpha_k}$ is a candidate of even universal sums of triangular numbers, then one of the following hold:
\begin{enumerate}
\item $\mathfrak{T}_e(\Delta_{1,1,7,\alpha_4,\alpha_5})=\infty$;
\item $\alpha_4=14$, $\alpha_5=\cdots=\alpha_{k-1}=35$ and $36\leq\alpha_k\leq40$;
\item $\alpha_4=14$, $\alpha_5=\cdots=\alpha_{k}=35$;
\item $\alpha_4=\cdots=\alpha_{k-1}=21$ and $22\leq\alpha_k\leq26$;
\item $\alpha_4=\cdots=\alpha_{k}=21$.
\end{enumerate}
In the first case, clearly $\Delta_{1,1,7,\alpha_4,\alpha_5,\dots,\alpha_k}$ is even universal but not proper.
In the second and fourth cases, since $\mathfrak{T}_e(\Delta_{1,1,7,\alpha_4,\alpha_5,\alpha_k})=\infty$, $\Delta_{1,1,7,\alpha_4,\alpha_5,\dots,\alpha_k}$ is even universal but not proper. 
In the third  case, $\mathfrak{T}_e(\Delta_{1,1,7,\alpha_4,\alpha_5,\dots,\alpha_k})=40$ since  $\mathfrak{T}(\Delta_{1,1,7,\alpha_4,\alpha_5,\dots,\alpha_{k-1}})=5$ and $\mathfrak{T}_e(\Delta_{1,1,7,\alpha_4,\alpha_5,\dots,\alpha_{k-1}})=40$.
Similarly, in the fifth case, one may easily check that $\mathfrak{T}_e(\Delta_{1,1,7,\alpha_4,\alpha_5,\dots,\alpha_k})=26$.
Therefore, if $k\geq 6$, then there are no $k$-ary proper even universal sums of triangular numbers.

\vskip1pc
\noindent\textbf{(iv)}  Let $(\alpha_1,\alpha_2,\alpha_3)=(2,2,3)$.  
For $k=4$, by the definition of the candidate of even universal sums of triangular numbers, we know that $3\leq \alpha_4 \leq 10$.
If $\alpha_4=4,8,10$, then $\Delta_{2,2,3,\alpha_4}$ is even universal but not proer.
In \cite{almost}, we proved that if $\alpha_4=5,7$, then $\Delta_{2,2,3,\alpha_4}$ represents all nonnegative integers except $1$.

Assume that $\alpha_4=3$.
We show that $\Delta_{2,2,3,3}$ is even universal.
By Equation \eqref{congruence condition}, it suffices to show that the equation
\begin{equation}\label{1166}
2x^2+2y^2+3z^2+3t^2 = 16n+10
\end{equation}
has an integer solution $(x,y,z,t)\in\mathbb{Z}^4$ such that  $xyzt\equiv1 \Mod2$ for any nonnegative integer $n$. 
Let $f(x,y,z)=2(4x+y)^2+2y^2+3z^2$.
In Section 3 of \cite{almost}, it was proved that for a nonnegative integer $m$, if $m\equiv7\Mod8$, $m\neq3^{2u+1}(3v+2)$ for any nonnegative integers $u$ and $v$, then $m$ is represented by $f$.
Let $16n+10=3^{2\ell}(16k+10)$ for some nonnegative integers $\ell$ and $k$ such that $16k\not\equiv0\Mod{9}$.
If $0\leq k\leq1$,  then the equation
$$
2x^2+2y^2+3z^2+3t^2=16k+10
$$
has an integer solution $(x,y,z,t)\in\mathbb{Z}^4$ such that $xyzt\equiv1\Mod2$.
Assume that $k\geq2$.
One may easily check that  $16k+10-3d^2$ is positive and represented by $f$, where 
$$
d=
\begin{cases}
3 \quad\text{if}~ 16k+10\equiv3 \Mod{9},\\
1 \quad\text{otherwise}.
\end{cases}
$$ 
Therefore,  the equation
$$
2x^2+2y^2+3z^2+3t^2=16k+10
$$
has an integer solution $(x,y,z,t)\in\mathbb{Z}^4$ such that $xyzt\equiv1\Mod2$.

Assume that $\alpha_4=6,9$.
Since $\mathfrak{T}_e(\Delta_{2,2,3,6})=16$ and $\mathfrak{T}_e(\Delta_{2,2,3,9})=10$ , $\Delta_{2,2,3,6}$ and $\Delta_{2,2,3,9}$ are not even universal.

For $k=5$, if $\Delta_{2,2,3,\alpha_4,\alpha_5}$ is a candidate of even universal sums of triangular numbers, then one of the following holds:
\begin{enumerate}
\item $\mathfrak{T}_e(\Delta_{2,2,3,\alpha_4})=\infty$;
\item $\alpha_4=6$ and $6\leq\alpha_5\leq16$;
\item $\alpha_4=9$ and $9\leq\alpha_5\leq10$.
\end{enumerate}
 In the first case, clearly $\Delta_{2,2,3,\alpha_4,\alpha_5}$ is even universal but not proper. 
 In the second case, if $\alpha_5=6,7,8,10,12,14,16$, then $\Delta_{2,2,3,6,\alpha_5}$ is even universal but not proper(see case (vi)). 
In \cite{almost}, we proved that if $\alpha_5=9,11,13$, then  $\Delta_{2,2,3,6,\alpha_5}$ represents all nonnegative integers except $1$. 
If $\alpha_5=15$, then one may easily check that $\mathfrak{T}_e(\Delta_{2,2,3,6,15})=16$.
In the third case, $\Delta_{2,2,3,9,10}$ is even universal but not proper.
One may easily check that $\mathfrak{T}_e(\Delta_{2,2,3,9,9})=10$.

For $k\geq 6$, similarly as in the proof of (iii), one may easily prove that there are no  $k$-ary proper even universal sums of triangular numbers.

\vskip1pc
\noindent\textbf{(v)}  Let $(\alpha_1,\alpha_2,\alpha_3)=(2,2,5)$.  
For $k=4$, by the definition of the candidate of even universal sums of triangular numbers, we know that $5\leq \alpha_4 \leq 10$.
If $\alpha_4=8,10$, then $\Delta_{2,2,5,\alpha_4}$ is even universal but not proer.

Assume that $\alpha_4=5$.
We show that $\Delta_{2,2,5,5}$ is even universal.
From \eqref{Liouville}, we know that $\Delta_{2,2,10}$ is even  universal. 
Therefore, $\Delta_{2,2,5,5}$ is even universal.

Assume that $\alpha_4=6,7,9$.
Since $\mathfrak{T}_e(\Delta_{2,2,5,6})=16$, $\mathfrak{T}_e(\Delta_{2,2,5,7})=10$ and $\mathfrak{T}_e(\Delta_{2,2,5,9})=10$ , $\Delta_{2,2,5,6}$, $\Delta_{2,2,5,7}$ and $\Delta_{2,2,5,9}$ are not even universal.

For $k=5$, if $\Delta_{2,2,5,\alpha_4,\alpha_5}$ is a candidate of even universal sums of triangular numbers, then one of the following holds:
\begin{enumerate}
\item $\mathfrak{T}_e(\Delta_{2,2,5,\alpha_4})=\infty$;
\item $\alpha_4=6$ and $6\leq\alpha_5\leq16$;
\item $\alpha_4=7,9$ and $\alpha_4\leq\alpha_5\leq10$.
\end{enumerate}
In the first case, clearly $\Delta_{2,2,5,\alpha_4,\alpha_5}$ is even universal but not proper. 
In the second case, If $\alpha_5=6,8,10,12,14,16$, then $\Delta_{2,2,5,6,\alpha_5}$ is even universal but not proper (see case (vi)). 
Now, we show that $\Delta_{2,2,5,6}$ represents all even integers except 16.
By Equation \eqref{congruence condition}, it suffices to show that the equation
\begin{equation}
2x^2+2y^2+5z^2+6t^2 = 16n+15
\end{equation}
has an integer solution $(x,y,z,t)\in\mathbb{Z}^4$ such that  $xyzt\equiv1 \Mod2$ for any nonnegative integer $n$.
Let $f(x,y,z)=2(2x+y)^2+2y^2+6t^2$.
From Table 14 of \cite{almost}, one may easily prove that for a nonnegative integer $m$, if $m\equiv10\Mod{16}$ and $m\neq3^{2u+1}(3v+1)$ for any nonnegative integers $u$ and $v$, then $m$ is represented by $f$.
If $0\leq n\leq24$ and $n\neq8$,  then the equation
$$
2x^2+2y^2+5z^2+6t^2=16n+15
$$
has an integer solution $(x,y,z,t)\in\mathbb{Z}^4$ such that $xyzt\equiv1\Mod2$.
Assume that $k\geq25$.
One may easily check that  $16n+15-5d^2$ is positive and represented by $f$, where 
$$
d=
\begin{cases}
1 \quad\text{if}~ 16n+15\equiv0 \Mod{3},\\
9 \quad\text{if}~ 16n+15\not\equiv0 \Mod{3}.
\end{cases}
$$ 
Therefore,  the equation
$$
2x^2+2y^2+5z^2+6t^2=16n+15
$$
has an integer solution $(x,y,z,t)\in\mathbb{Z}^4$ such that $xyzt\equiv1\Mod2$ for any nonnegative integer $n$ except 8.
Therefore, If $\alpha_5=7,9,11$, then one may easily show that $\Delta_{2,2,5,6,\alpha_5}$ is even universal.
If $\alpha_5=13,15$, then one may easily check that $\mathfrak{T}_e(\Delta_{2,2,5,6,13})=\mathfrak{T}_e(\Delta_{2,2,5,6,15})=16$.
In the third case, one may easily show that $\Delta_{2,2,5,\alpha_4,\alpha_5}$ is even universal but not proper or $\mathfrak{T}_e(\Delta_{2,2,5,\alpha_4,\alpha_5})=10$.


For $k\geq 6$, similarly as in the proof of (iii), one may easily prove that there are no  $k$-ary proper even universal sums of triangular numbers.
\vskip1pc
\noindent\textbf{(vi)} Let $(\alpha_1,\alpha_2,\alpha_3)=(2,2,6)$. For $k=4$, by the definition of the candidate of even universal sums of triangular numbers, we know that $6\leq \alpha_4 \leq 16$.
If $\alpha_4=8,10$, then $\Delta_{2,2,6,\alpha_4}$ is even universal but not proer.
If $\alpha_4=6,12,14,16$, then one may easily show that $\Delta_{2,2,6,\alpha_4}$ is even universal by the triangular theorem of eight of Bosma and Kane. 
If $\alpha_4=7,9,11,13,15$, then $\mathfrak{T}_e(\Delta_{2,2,6,\alpha_4})=16$.

For $k=5$, if $\Delta_{2,2,6,\alpha_4,\alpha_5}$ is a candidate of even universal sums of triangular numbers, then one of the following holds:
\begin{enumerate}
\item $\mathfrak{T}_e(\Delta_{2,2,6,\alpha_4})=\infty$;
\item $\alpha_4=7$ and $7\leq\alpha_5\leq16$;
\item $\alpha_4=9,11,13,15$ and $\alpha_4\leq\alpha_5\leq16$.
\end{enumerate}
 In the first case, clearly $\Delta_{2,2,6,\alpha_4}$ is even universal but not proper. 
 In the second case, if $\alpha_5=8,10,12,14,16$, then $\Delta_{2,2,6,7,\alpha_5}$ is even universal but not proper.
 If $\alpha_5=7,9$, then  $\Delta_{2,2,6,7,7}$ and  $\Delta_{2,2,6,7,9}$ are even universal since  $\Delta_{2,2,6,14}$ and  $\Delta_{2,2,6,16}$ are even universal.
 If $\alpha_5=11,13,15$, then $\mathfrak{T}_e(\Delta_{2,2,6,7,\alpha_5})=16$.
In the third case, one may easily show that $\Delta_{2,2,6,\alpha_4,\alpha_4}$ is even universal but not proper or $\mathfrak{T}_e(\Delta_{2,2,6,\alpha_4,\alpha_4})=16$.

For $k\geq 6$, similarly as in the proof of (iii), one may easily prove that there are no  $k$-ary proper even universal sums of triangular numbers.

\vskip1pc
\noindent\textbf{(vii)} Let $(\alpha_1,\alpha_2,\alpha_3)=(2,2,7)$. For $k=4$, by the definition of the candidate of even universal sums of triangular numbers, we know that $7\leq \alpha_4 \leq 10$.
If $\alpha_4=8,10$, then $\Delta_{2,2,7,\alpha_4}$ is even universal but not proer.
If $\alpha_4=7,9$, then one may easily check that$\mathfrak{T}_e(\Delta_{2,2,7,7})=\mathfrak{T}_e(\Delta_{2,2,7,9})=10$.
So there are no quaternary proper even universal sums of triangular numbers.

For $k\geq 5$, similarly as in the proof of (ii), one may easily prove that there are no  $k$-ary proper even universal sums of triangular numbers.

\vskip1pc
\noindent\textbf{(viii)} Let $(\alpha_1,\alpha_2,\alpha_3)=(2,2,9)$. For $k=4$, by the definition of the candidate of even universal sums of triangular numbers, we know that $9\leq \alpha_4 \leq 10$.
One may easily check that $\Delta_{2,2,9,10}$ is even universal but not proper and $\mathfrak{T}_e(\Delta_{2,2,9,9})=10$.
So there are no quaternary proper even universal sums of triangular numbers.

For $k\geq 5$, similarly as in the proof of (ii), one may easily prove that there are no  $k$-ary proper even universal sums of triangular numbers.

\vskip1pc
\noindent\textbf{(ix)} Let $(\alpha_1,\alpha_2,\alpha_3)=(2,3,3)$. For $k=4$, by the definition of the candidate of even universal sums of triangular numbers, we know that $3\leq \alpha_4 \leq 4$.
In \cite{almost}, we proved that $\Delta_{2,3,3,4}$ is almost universal with one exception 1. 
One may easily check that $\mathfrak{T}_e(\Delta_{2,3,3,3})=4$.

For $k\geq 5$, similarly as in the proof of (ii), one may easily prove that there are no  $k$-ary proper even universal sums of triangular numbers.

\vskip1pc
\noindent\textbf{(x)} Let $(\alpha_1,\alpha_2,\alpha_3)=(2,3,4)$. For $k=4$, by the definition of the candidate of even universal sums of triangular numbers, we know that $4\leq \alpha_4 \leq 8$.
If $\alpha_4=4,6,8$, then $\Delta_{2,3,4,\alpha_4}$ is even universal but not proper.
In \cite{almost}, we proved that $\Delta_{2,3,4,5}$ is almost universal with one exception 1. 
One may easily check that $\mathfrak{T}_e(\Delta_{2,3,4,7})=8$.

For $k\geq 5$, similarly as in the proof of (ii), one may easily prove that there are no  $k$-ary proper even universal sums of triangular numbers.

\vskip1pc
\noindent\textbf{(xi)} Let $(\alpha_1,\alpha_2,\alpha_3)=(2,4,5)$. For $k=4$, by the definition of the candidate of even universal sums of triangular numbers, we know that $5\leq \alpha_4 \leq 8$.
If $\alpha_4=6,8$, then $\Delta_{2,4,5,\alpha_4}$ is even universal but not proper.
One may easily check that $\mathfrak{T}_e(\Delta_{2,4,5,5})=\mathfrak{T}_e(\Delta_{2,4,5,7})=8$.
So there are no quaternary proper even universal sums of triangular numbers.

For $k\geq 5$, similarly as in the proof of (ii), one may easily prove that there are no  $k$-ary proper even universal sums of triangular numbers.

\vskip1pc
\noindent\textbf{(xii)} Let $(\alpha_1,\alpha_2,\alpha_3)=(2,4,7)$. For $k=4$, by the definition of the candidate of even universal sums of triangular numbers, we know that $7\leq \alpha_4 \leq 8$.
One may easily check that $\Delta_{2,4,7,8}$ is even universal but not proper and $\mathfrak{T}_e(\Delta_{2,4,7,7})=8$.
So there are no quaternary proper even universal sums of triangular numbers.

For $k\geq 5$, similarly as in the proof of (ii), one may easily prove that there are no  $k$-ary proper even universal sums of triangular numbers.

\begin{table}[hpt!]
\caption{Proper even universal sums of triangular numbers}
\renewcommand{\arraystretch}{1}\renewcommand{\tabcolsep}{1mm}
\begin{tabular}{l|ll}
\hline
Sums & Conditions on $\alpha_k$ &$(3\leq k \leq5)$ \\
\hline\hline
$\Delta_{1,1,\alpha_3}$ & $1\leq\alpha_3\leq8$, &$\alpha_3\neq3,6,7$ \\ \hline
$\Delta_{1,2,\alpha_3}$ & $2\leq\alpha_3\leq4$ &\\ \hline
$\Delta_{2,2,\alpha_3}$ & $2\leq\alpha_3\leq10$, &$\alpha_3\neq3,5,6,7,9$ \\ \hline
$\Delta_{2,4,\alpha_3}$ & $4\leq\alpha_3\leq8$, &$\alpha_3\neq5,7$ \\ \hline\hline
$\Delta_{1,1,3,\alpha_4}$ & $3\leq\alpha_4\leq8$, &$\alpha_4\neq4,5,8$ \\ \hline
$\Delta_{1,1,6,\alpha_4}$ & $6\leq\alpha_4\leq14$, &$\alpha_4\neq8,9$ \\ \hline
$\Delta_{1,1,7,\alpha_4}$ & $7\leq\alpha_4\leq26$, &$\alpha_4\neq8,14,21$ \\ \hline
$\Delta_{2,2,3,\alpha_4}$ & $3\leq\alpha_4\leq10$, &$\alpha_4\neq4,6,8,9,10$ \\ \hline
$\Delta_{2,2,5,\alpha_4}$ & $5\leq\alpha_4\leq10$, &$\alpha_4\neq6,7,8,9,10$ \\ \hline
$\Delta_{2,2,6,\alpha_4}$ & $6\leq\alpha_4\leq16$, &$\alpha_4\neq7,8,9,10,11,13,15$ \\ \hline
$\Delta_{2,3,3,\alpha_4}$ & $3\leq\alpha_4\leq4$, &$\alpha_4\neq3$ \\ \hline
$\Delta_{2,3,4,\alpha_4}$ & $4\leq\alpha_4\leq8$, &$\alpha_4\neq4,6,7,8$ \\ \hline\hline
$\Delta_{1,1,7,14,\alpha_5}$ & $14\leq\alpha_5\leq40$, &$\alpha_5\neq15,16,17,18,19,20,22,23,24,25,26,35$ \\ \hline
$\Delta_{2,2,3,6,\alpha_5}$ & $6\leq\alpha_5\leq16$, &$\alpha_5\neq6,7,8,10,12,14,15,16$ \\ \hline
$\Delta_{2,2,5,6,\alpha_5}$ & $6\leq\alpha_5\leq16$, &$\alpha_5\neq6,8,10,12,13,14,15,16$ \\ \hline
$\Delta_{2,2,6,7,\alpha_5}$ & $7\leq\alpha_5\leq16$, &$\alpha_5\neq8,10,11,12,13,14,15,16$ \\ \hline

\end{tabular}
\end{table}

\begin{proof}[Proof of Theorem \ref{1}]
Assume that a sum $\Delta_{\alpha_1,\alpha_2,\dots,\alpha_k}$ of triangular numbers represents the integers
$$
2,~4,~8,~10,~14,~16,~26,~\text{and}~40.
$$
By using the same escalation method to the above, one may easily show that $k\geq3$ and  
there is a subset $\{\alpha_{i_1},\dots,\alpha_{i_\ell}\}$ of $\{\alpha_1,\dots,\alpha_k\}$ with $3\leq\ell\leq5$ such that $\Delta_{\alpha_{i_1},\dots,\alpha_{i_\ell}}$ is contained in 75  proper even universal sums of triangular numbers in Table 1.
Therefore, $\Delta_{\alpha_1,\dots,\alpha_k}$ is even universal.
This completes the proof.

\end{proof}


\end{document}